\documentclass{amsart}
\usepackage[dvipdfmx]{graphicx}
\usepackage{tikz}
\usepackage{amssymb,amsmath,amsthm,amscd,amsfonts}

\newtheorem{theorem}{Theorem}[section]
\newtheorem{proposition}[theorem]{Proposition}

\newtheorem{lemma}[theorem]{Lemma}

\theoremstyle{definition}
\newtheorem{definition}[theorem]{Definition}
\newtheorem{remark}[theorem]{Remark}

\newtheorem*{OTP*}{Organization of this paper}
\newtheorem*{Ack*}{Acknowledgements}

\begin{document}


\title[]
{Special Lagrangian and deformed Hermitian Yang--Mills on tropical manifold}


\author{Hikaru Yamamoto}
\address{Department of Mathematics, Faculty of Science, Tokyo University of Science, 1-3 Kagurazaka, Shinjuku-ku, Tokyo 162-8601, Japan}
\email{hyamamoto@rs.tus.ac.jp}




\begin{abstract}
From string theory, the notion of deformed Hermitian Yang--Mills connections has been introduced 
by Mari\~no, Minasian, Moore and Strominger \cite{MarinoMinasianMooreStrominger}. 
After that, Leung, Yau and Zaslow \cite{LeungYauZaslow} proved that 
it naturally appears as mirror objects of special Lagrangian submanifolds via Fourier--Mukai transform between dual torus fibrations. 
In their paper, some conditions are imposed for simplicity. 
In this paper, data to glue their construction on tropical manifolds are proposed and 
a generalization of the correspondence is proved 
without the assumption that the Lagrangian submanifold is a section of the torus fibration. 
\end{abstract} 






\thanks{This work was supported by JSPS KAKENHI Grant Number 16H07229}



\maketitle

\section{Introduction}\label{intro}
Leung, Yau and Zaslow \cite{LeungYauZaslow} studied 
{\it deformed Hermitian Yang--Mills connections}. 
From string theory, the notion had been introduced by Mari\~no, Minasian, Moore and Strominger \cite{MarinoMinasianMooreStrominger}. 
Let $W$ be a K\"ahler manifold with a K\"ahler form $\tilde{\omega}$ and 
$\mathcal{L}\to W$ be a complex line bundle  with a Hermitian metric $h$. 

\begin{definition}
Fix a constant $\theta\in\mathbb{R}$. 
A Hermitian connection $D$ on $\mathcal{L}$ is called a deformed Hermitian Yang--Mills connection with phase $e^{\sqrt{-1}\theta}$ if 
\begin{equation}\label{defofdHYM}
F^{(0,2)}=0\quad\text{and}\quad \mathop{\text{Im}}\left(e^{-\sqrt{-1}\theta}(\tilde{\omega}+F)^{m}\right)=0,
\end{equation}
where $F^{(0,2)}$ is the $(0,2)$-part of the curvature 2-form $F$ of $D$ and 
$m$ is the complex dimension of $W$. 
\end{definition}

The former condition of \eqref{defofdHYM} is equivalent to that $D$ defines a holomorphic structure on $\mathcal{L}$, and 
such a connection is said to be {\it integrable}. 
Additionally, when $m=2$ and $\theta=0$, the latter condition coincides with 
$\tilde{\omega}\wedge F=0$ by a simple computation 
$(\tilde{\omega}+F)^2=(\tilde{\omega}^2+F^2)+2\tilde{\omega}\wedge F$ and the fact that $F$ is purely imaginary. 
This is just the ordinary definition of Hermitian Yang--Mills connections. 
However, for higher dimensional case, it is different from that, 
and is more complicated, especially non-linear. 

In the paper \cite{LeungYauZaslow}, it is proved that the deformed Hermitian Yang-Mills connections naturally appear 
as mirror objects of special Lagrangian submanifolds via Fourier--Mukai transform 
between dual torus fibrations. 
A {\it special Lagrangian submanifold} is a half-dimensional submanifold in 
an almost Calabi--Yau manifold $X$ with a K\"ahler form $\omega$ and 
a holomorphic volume form $\Omega$ which is not necessary to be parallel. 

\begin{definition}
Fix a constant $\theta\in\mathbb{R}$. 
A real $m$ dimensional submanifold $L$ in $X$ is called 
a special Lagrangian submanifold with phase $e^{\sqrt{-1}\theta}$ if 
\begin{equation}\label{defofsLag}
\omega|_{L}=0\quad\text{and}\quad \mathop{\text{Im}}\left(e^{-\sqrt{-1}\theta}\Omega\right)|_{L}=0,
\end{equation}
where $m$ is the complex dimension of $X$. 
\end{definition}

The former condition of \eqref{defofsLag} requires $L$ to be a Lagrangian submanifold. 
With the latter condition, $L$ becomes a calibrated submanifold with respect to $\mathop{\text{Re}}(e^{-\sqrt{-1}\theta}\Omega)$ 
as the original definition by Harvey and Lawson \cite{HarveyLawson}. 
Especially, $L$ is volume minimizing in its homology class. 
Special Lagrangian submanifolds acquired physical importance in the context of mirror symmetry by 
Strominger, Yau and Zaslow \cite{StromingerYauZaslow}. 

Let us return to the brief introduction of the work of Leung, Yau and Zaslow \cite{LeungYauZaslow}. 
In their paper, two conditions are imposed for simplicity. 
The one is for ambients and the another is for submanifolds. 
For ambient spaces $X$ and $W$, they assume the following. 
Let $B$ be an open subset in $\mathbb{R}^{m}$ and $\phi$ be a smooth convex function on $B$. 
Define $X$ by $X=TB/\Lambda$, where $\Lambda=B\times\mathbb{Z}^{m}$ is the lattice bundle in the tangent bundle $TB=B\times\mathbb{R}^{m}$. 
Then, this gives a trivial torus fibration $f:X\to B$. 
Let $y_{j}$ be the standard coordinates on each torus fiber $T^{m}:=f^{-1}(x)$. 
Then, a complex structure on $X$ is given by coordinates $z_{i}:=x_{i}+\sqrt{-1}y_{i}$. 
Moreover, a holomorphic volume form $\Omega$ and a K\"ahler form $\omega$ on $X$ are defined by
$\Omega:=dz_{1}\wedge\cdots\wedge dz_{m}$ and 
\[\omega:=\frac{\sqrt{-1}}{2}\sum_{i,j=1}^{m}\phi_{ij}dz_{i}\wedge d\bar{z}_{j}, \]
where $\phi_{ij}$ is the Hessian matrix of $\phi$. 
Then, $(X,\omega,\Omega)$ is an almost Calabi--Yau manifold. 
In this setting, we get another almost Calabi--Yau manifold $W$ 
as $W=T^{*}B/\Lambda^{*}$, where $\Lambda^{*}=B\times(\mathbb{Z}^{m})^{*}$ is the dual lattice bundle, with a trivial torus fibration $\tilde{f}:W\to B$. 
Put $\tilde{x}_{i}:=\partial \phi/\partial x_{i}$, the {\it Legendre transform} of $\phi$. 
Denote the standard coordinates on each dual torus fiber $\tilde{T}^{m}:=\tilde{f}^{-1}(x)$ by $\tilde{y}_{i}$ and 
the inverse matrix of $\phi_{ij}$ by $\phi^{ij}$. 
Then, a complex structure on $W$ is given by coordinates $\tilde{z}_{i}:=\tilde{x}_{i}+\sqrt{-1}\tilde{y}_{i}$, 
and we get a holomorphic volume form $\tilde{\Omega}:=d\tilde{z}_{1}\wedge\cdots\wedge d\tilde{z}_{m}$
and a K\"ahler form 
\[\tilde{\omega}:=\frac{\sqrt{-1}}{2}\sum_{i,j=1}^{m}\phi^{ij}d\tilde{z}_{i}\wedge d\bar{\tilde{z}}_{j}\] 
on $W$. Clearly, $(W,\tilde{\omega},\tilde{\Omega})$ is also an almost Calabi--Yau manifold. 
This situation is called the {\it semi-flat} case, and the pair $(X,\omega,\Omega)$ and $(W,\tilde{\omega},\tilde{\Omega})$ 
is called a {\it mirror pair} or {\it dual torus fibrations}. 
For submanifolds, they assume that $L$ in $X$ is written as a graph of a section $Y$ of the torus fibration $f:X\to B$. 
We denote it by $L(Y):=\{\,(x,Y(x))\in X\mid x\in B\,\}$. 

In this setting, their idea and theorem are as follows. 
At each point $x\in B$, we get a point $Y(x) \in T^{m}=f^{-1}(x)$, 
and identify $T^{m}$ with $\mathop{\text{Hom}}(\pi_{1}(\tilde{T}^{m}),S^{1})$ canonically. 
Then, by the correspondence between monodromy representations and flat bundles, 
$Y(x)$ determines a flat Hermitian connection, denoted by $D^{Y(x)}$, on 
the trivial line bundle $\underline{\mathbb{C}}\to \tilde{T}^{m}=\tilde{f}^{-1}(x)$ with the standard Hermitian metric. 
Then, the family $D^{Y(x)}$ along $x\in B$ gives 
a Hermitian connection (which lacks the flatness), denoted by $D^{Y}$, 
on the trivial line bundle $\underline{\mathbb{C}}\to W$ with the standard Hermitian metric. 
Explicitly, it is written as 
\[D^{Y}=d+\sqrt{-1}\sum_{j-1}^{m}Y^{j}d\tilde{y}_{j}.\]
Then, they proved the following equivalence. 

\begin{theorem}[\cite{LeungYauZaslow}]\label{LYZthm}
$L(Y)$ is a Lagrangian submanifold in $(X,\omega)$ if and only if $D^{Y}$ is an integrable connection 
on $\underline{\mathbb{C}}\to W$. 
Furthermore, $L(Y)$ is a special Lagrangian submanifold in $(X,\omega,\Omega)$ with phase $e^{\sqrt{-1}\theta}$ 
if and only if $D^{Y}$ is a deformed Hermitian Yang--Mills connection on $\underline{\mathbb{C}}\to (W,\tilde{\omega})$ with phase $e^{\sqrt{-1}\theta}$. 
\end{theorem}

Since $L=L(Y)$ is a graph of a section, 
$f(L(Y))$ (the shadow of $L(Y)$ by the projection $f:X\to B$) is $B$ itself. 
Hence, the dimension of it is $\dim_{\mathbb{R}}B$. 
For the case when the dimension of the shadow of $L\subset X$, denoted by $k$, is less than $\dim_{\mathbb{R}}B$, 
Leung, Yau and Zaslow \cite{LeungYauZaslow} explain that 
we get a complex submanifold $C\subset W$ with complex dimension $k$ and 
a deformed Hermitian Yang--Mills connection on a trivial Hermitian line bundle $\underline{\mathbb{C}}$ over $C$, not $W$, 
by explicit calculation when $m=3$. 
For general $m$ and $k$, Chan \cite{Chan} explains 
how to generalize the construction of $D^{Y}$ from $L(Y)$, 
and prove the equivalence of the conditions of Lagrangian submanifolds and integrable connections. 
Furthermore, the base space $B$ treated in \cite{Chan} is a tropical  manifold. 

In this paper, as Chan has done, 
we take tropical  manifolds as base spaces and 
we glue the local construction due to Leung, Yau and Zaslow \cite{LeungYauZaslow} together 
on the tropical manifold. 
The way to glue is similar to the one to construct a global section of a sheaf 
from local sections with compatibility. 
Some basic calculation have been done in \cite{Chan} however, 
this paper has two improvements. 
One is that the data to construct submanifolds and connections are explicitly mentioned. 
The another is that we prove the equivalence 
of the conditions of special Lagrangian submanifolds and deformed Hermitian Yang--Mills connections in this setting. 

In the following, we state the main theorem in this paper omitting some definitions. 
Let $(B,\mathcal{D}, K)$ be an $m$-dimensional tropical  manifold $B$ 
with a tropical structure $\mathcal{D}=\{\,(U_{\lambda},\psi_{\lambda})\,\}_{\lambda\in\Lambda}$ and 
a convex multi-valued function $K=\{\,K_{\lambda}\,\}_{\lambda\in\Lambda}$ on $B$. 
Then, we get the associated mirror almost Calabi--Yau pair $(X(B),\omega,\Omega)$ and $(W(B),\tilde{\omega},\tilde{\Omega})$, 
as in Section \ref{setting}. 
Fix an integer $k$ with $0\leq k \leq m$. 
Then, we get a fiber bundle $\mathcal{F}(k,m)\to B$ 
such that the fiber is the set of $k$-dimensional rational affine linear subspace in $\mathbb{R}^{m}$. 
Fix a constant section $V$ of $\mathcal{F}(k,m)\to B$. 
We assume the existence of such $V$. 
The definition of $\mathcal{F}(k,m)$ and the {\it constant section} are given in Section \ref{bundle}. 
From $V$, we construct a $k$-dimensional submanifold $B(V)$ in $B$, as explained in Section \ref{cdata}. 
Fix a vector field $Y$ on $B(V)$. 
We call these (inductively taken) data $(k,V,Y)$ the {\it construction data}. 
In Section \ref{LAG}, we give an explicit way to construct 
an $m$-dimensional submanifold $L(V,Y)$ in $X(B)$ from $(k,V,Y)$ such that 
the dimension of the image of it by $f:X(B)\to B$ is $k$. 
In Section \ref{ccc}, we give an explicit way to construct 
a complex submanifold $C(V)$ in $W(B)$ with complex dimension $k$ and 
a Hermitian connection $D^{Y}$ on the trivial line bundle $\underline{\mathbb{C}}$ over $C(V)$ with the standard Hermitian metric. 
We give a K\"ahler structure on $C(V)$ by the restriction of $\tilde{\omega}$. 
Then, the main theorem of this paper is the following. 

\begin{theorem}\label{main}
$L(V,Y)$ is a Lagrangian submanifold in $(X(B),\omega)$ if and only if $D^{Y}$ is an integrable connection 
on $\underline{\mathbb{C}}\to C(V)$. 
Furthermore, $L(V,Y)$ is a special Lagrangian submanifold in $(X(B),\omega,\Omega)$ with phase $e^{\sqrt{-1}\theta}$ 
if and only if $D^{Y}$ is a deformed Hermitian Yang--Mills connection on $\underline{\mathbb{C}}\to (C(V),\tilde{\omega})$ with phase $e^{\sqrt{-1}\theta}$. 
\end{theorem}

The proof will be done by converting each property of $L(V,Y)$ and $(D^{Y},C(V))$ to one of $(k,V,Y)$. 
The correspondence is summarized as Table \ref{tab2}. 
The first line (the triviality) is proved by Proposition \ref{normalprop} and \ref{lem2}. 
The middle and final line are just consequences from Proposition \ref{Lagpt} and \ref{integpt}, and Proposition \ref{sLagpt} and \ref{dHYMprop}, 
respectively. 

\renewcommand{\arraystretch}{1.3}
\begin{table}[h]
\caption{Correspondence}
\label{tab2}
\begin{tabular}{|c|c|c|}
\hline
 $L(V,Y)$ in $X(B)$ is & Data $(k,V,Y)$ & $D^{Y}$ on $C(V)$ ($\subset W(B)$) is\\
 \hline \hline
  trivial & $Y$ is normal & trivial  \\
\hline
Lagrangian & $Y$ is locally gradient & integrable  \\
\hline
special Lagrangian & $Y$ satisfies \eqref{MA} & deformed Hermitian Yang--Mills\\
\hline
\end{tabular}
\end{table}
\renewcommand{\arraystretch}{1}

To state our main theorem, we assumed the existence of a constant section $V$ of $\mathcal{F}(k,m)\to B$. 
When $k=m$, then $\mathcal{F}(m,m)=B\times \{\,\mathbb{R}^{m}\,\}$, and 
there exists only one constant section $V$ of $\mathcal{F}(m,m)\to B$, that is, $V(x)=\mathbb{R}^{m}$. 
For such $V$, we see that $B(V)$ becomes $B$ itself and 
Theorem \ref{main} is reduced to Theorem \ref{LYZthm}. 
However, when $k<m$, the existence of constant sections of $\mathcal{F}(k,m)\to B$ is not assured. 
It might depend on the global geometry of $(B,\mathcal{D}, K)$, 
the tropical  manifold with a convex multi-valued function. 

Finally, we mention some related studies of deformed Hermitian Yang--Mills connections, 
though it is relatively rare. 
For example, Jacob and Yau \cite{JacobYau} considered an analogue of Lagrangian mean curvature flows for Hermitian connections. 
Collins, Jacob and Yau \cite{CollinsJacobYau} gave a priori estimates and proved the existence of 
deformed Hermitian Yang--Mills connections under some assumptions and 
found some obstructions to the existence. 
By using a different method from \cite{CollinsJacobYau}, 
Pingali \cite{Pingali} gave similar a priori estimates for deformed Hermitian Yang--Mills connections. 
More recently, Collins, Xie and Yau \cite{CollinsXieYau} provided an introduction to deformed Hermitian Yang--Mills connections and 
gave a new Chern number inequality under the assumption that a deformed Hermitian Yang--Mills connection exists. 

\begin{OTP*}
This section, Section \ref{intro}, is an introduction to main subjects in this paper. 
The notions of deformed Hermitian Yang--Mills connections 
and special Lagrangian submanifolds are defined in this section. 
The idea of Leung, Yau and Zaslow \cite{LeungYauZaslow} and 
their theorem are reviewed here, 
and the main theorem of this paper is also given. 
Section \ref{setting} is just a brief introduction to the tropical geometry which is a basic material of this paper. 
Tropical manifolds and associated dual torus fibrations are explained here. 
In Section \ref{bundle}, a fiber bundle on a tropical manifold with a convex multi-valued function 
is defined. A section of this bundle is one of important ingredients of our construction. 
In section \ref{cdata}, the data to construct a special Lagrangian submanifold and deformed Hermitian Yang--Mills connection is introduced. 
In Section \ref{LAG}, we give a procedure to construct a Lagrangian submanifold from the data and 
give a condition such that it is a special Lagrangian submanifold. 
Section \ref{ccc} is the mirror side argument of Section \ref{LAG}. 
A complex submanifold and connection over it are constructed, and 
a condition such that it is a deformed Hermitian Yang--Mills connection is given. 
At the end of this section, we give a proof of Theorem \ref{main}. 
\end{OTP*}

\begin{Ack*}
This work was based on private communications with Professor A. Futaki, 
so the author would like to thank him for many suggestions and stimulating discussions. 
\end{Ack*}
\section{Setting}\label{setting}
In this section, following \cite{Gross}, we explain some terminologies used in semi-flat mirror symmetry, 
the toy version of mirror symmetry. 

\begin{definition}
A {\it tropical  manifold} is a manifold $B$ with a differential structure 
$\mathcal{D}=\{\,(U_{\lambda},\psi_{\lambda})\,\}_{\lambda\in\Lambda}$ 
such that transition functions $\psi_{\lambda}\circ \psi_{\mu}^{-1}$ 
lie in $\mathbb{R}^{m}\rtimes GL(\mathbb{Z}^{m})$. 
The differential structure $\mathcal{D}$ is called a tropical structure and 
local coordinates $(x_{1},\dots,x_{m})$ on $U_{\lambda}$ defined by $\psi_{\lambda}:U_{\lambda}\to\mathbb{R}^{m}$ 
are called tropical coordinates. 
\end{definition}

\begin{definition}
Let $B$ be a tropical  manifold with a tropical structure $\mathcal{D}=\{\,(U_{\lambda},\psi_{\lambda})\,\}_{\lambda\in\Lambda}$. 
A {\it multi-valued function} on $B$ is a collection of functions $K=\{\,K_{\lambda}\,\}_{\lambda\in\Lambda}$ 
such that each $K_{\lambda}$ is a smooth function on $U_{\lambda}$ and 
$K_{\lambda}-K_{\mu}$ is affine linear on $U_{\lambda}\cap U_{\mu}$ with respect to tropical coordinates. 
A multi-valued function $K$ is {\it convex} if the Hessian of $K_{\lambda}$ with respect to tropical coordinates is positive definite on $U_{\lambda}$ for all $\lambda\in\Lambda$.
\end{definition}

\begin{definition}
A Riemannian metric $g$ on a tropical  manifold $B$ 
with a tropical structure $\mathcal{D}=\{\,(U_{\lambda},\psi_{\lambda})\,\}_{\lambda\in\Lambda}$ is {\it Hessian} if 
there exists a convex multi-valued function $K=\{\,K_{\lambda}\,\}_{\lambda\in\Lambda}$ on $B$ such that 
\[g\left(\frac{\partial}{\partial x_{i}},\frac{\partial}{\partial x_{j}}\right)=\frac{\partial^2 K_{\lambda}}{\partial x_{i}\partial x_{j}}\]
on $U_{\lambda}$ for all $\lambda\in\Lambda$. 
\end{definition}

Let $(B,\mathcal{D}, K)$ be an $m$-dimensional tropical  manifold $B$ 
with a tropical structure $\mathcal{D}=\{\,(U_{\lambda},\psi_{\lambda})\,\}_{\lambda\in\Lambda}$ and 
a convex multi-valued function $K=\{\,K_{\lambda}\,\}_{\lambda\in\Lambda}$ on $B$. 
We denote by $g$ the Riemannian metric on $B$ associated with $K$. 
In what follows, we call the triplet $(B,\mathcal{D}, K)$ a tropical  manifold for short. 
Then, its {\it Legendre transform} is defined as follows. 
First, we take $B$ itself as the underlying differential manifold. 
Next, we define another tropical coordinates $(\tilde{x}_{1},\dots,\tilde{x}_{m})$ on $U_{\lambda}$ by 
\[\tilde{x}_{i}=\tilde{x}_{i}(x):=\frac{\partial K_{\lambda}}{\partial x_{i}}, \]
that is, the Legendre transform of $K_{\lambda}$, where 
$(x_{1},\dots,x_{m})$ are the original tropical coordinates on $U_{\lambda}$. 
Let $x'_{i}$ and $\tilde{x}'_{i}$ be the original tropical coordinates and its Legendre transforms on $U_{\mu}$, respectively. 
Since $K_{\lambda}-K_{\mu}$ is affine linear, there exist real constants $a_{i}$ and $c$ such that 
$K_{\lambda}=K_{\mu}+\sum_{i=1}^{m}a_{i}x_{i}+c$. 
Then, we have
\begin{equation}\label{change}
\tilde{x}_{i}=\frac{\partial K_{\lambda}}{\partial x_{i}}=\frac{\partial K_{\mu}}{\partial x_{i}}+a_{i}
=\sum_{j=1}^{m}\frac{\partial K_{\mu}}{\partial x'_{j}}\frac{\partial x'_{j}}{\partial x_{i}}+a_{i}
=\sum_{j=1}^{m}\frac{\partial x'_{j}}{\partial x_{i}}\tilde{x}'_{j}+a_{i}. 
\end{equation}
This together with $(\partial x'_{j}/\partial x_{i})_{ij}\in GL(\mathbb{Z}^{m})$ implies 
that $\mathcal{\tilde{D}}:=\{\,(U_{\lambda},\tilde{\psi}_{\lambda})\,\}_{\lambda\in\Lambda}$ 
with $\tilde{\psi}_{\lambda}=(\tilde{x}_{1},\dots,\tilde{x}_{m})$ is also a tropical structure on $B$. 
Define the function $\tilde{K}_{\lambda}$ on $U_{\lambda}$ by
\[\tilde{K}_{\lambda}(\tilde{x}):=\sum_{i=1}^{m}\tilde{x}_{i}x_{i}-K_{\lambda}(x). \]
Then, one can easily prove that $\tilde{K}:=\{\,\tilde{K}_{\lambda}\,\}_{\lambda\in\Lambda}$ is 
a convex multi-valued function on $B$ with respect to the tropical structure $\mathcal{\tilde{D}}$. 
We call the triplet $(B,\mathcal{\tilde{D}}, \tilde{K})$ the Legendre transform of $(B,\mathcal{D}, K)$. 

Let $(B,\mathcal{D}, K)$ be an $m$-dimensional tropical manifold. 
For a chart $U_{\lambda}$, a module $\Lambda|_{U_{\lambda}}$ over $\mathbb{Z}$ is defined by 
\[\Lambda|_{U_{\lambda}}:=\mathbb{Z}\frac{\partial}{\partial x_{1}}+\cdots+\mathbb{Z}\frac{\partial}{\partial x_{m}}.\]
Since $B$ is tropical, $\Lambda|_{U_{\lambda}}$ and $\Lambda|_{U_{\mu}}$ are compatible over $U_{\lambda}\cap U_{\mu}$. 
Thus, the lattice bundle $\Lambda\subset TB$ is defined by gluing these together. 
Then, the complex manifold $X(B)$ is defined by
\[X(B):=TB/\Lambda\]
with holomorphic coordinates $z_{i}:=x_{i}+\sqrt{-1}y_{i}$, where $y_{i}$ is the coefficient of $\partial/\partial x_{i}$ of an element in $T_{x}B$. 
We denote the projection map by $f:X(B)\to B$ and it is clear that each fiber is an $m$-dimensional torus $T^{m}$. 
Hence, $f:X(B)\to B$ is a torus fibration over $B$. 
Since $B$ is tropical, a holomorphic $(m,0)$-form defined by
\[\Omega:=dz_{1}\wedge \dots \wedge dz_{m}\]
is well-defined and nonvanishing. 
By using $K$ as potential functions, a K\"ahler form $\omega$ on $X(B)$ is defined by
\begin{equation}\label{symp}
\omega:=\frac{\sqrt{-1}}{2}\sum_{i,j}\frac{\partial^2 K_{\lambda}}{\partial x_{i}\partial x_{j}}dz_{i}\wedge d\overline{z}_{j}
\end{equation}
on each $U_{\lambda}$ and these are compatible on $U_{\lambda}\cap U_{\mu}$. 
One can easily check that 
\[\frac{\omega^{m}}{m!}=\left(\frac{\sqrt{-1}}{2}\right)^{m}(-1)^{\frac{m(m-1)}{2}}
\det\left(\frac{\partial^2 K_{\lambda}}{\partial x_{i}\partial x_{j}}\right)\Omega\wedge\bar{\Omega}. \]
Thus, $(X(B),\omega)$ is Ricci-flat if and only if every $K_{\lambda}$ satisfies 
\[\det \left(\frac{\partial^2 K_{\lambda}}{\partial x^{i}\partial x^{j}}\right)=\text{const},\]
the so-called real Monge--Amp\`ere equation. 
In this paper, we do not assume that $(X(B),\omega)$ is Ricci-flat. 
A $2m$-dimensional K\"ahler manifold with a nonvanishing holomorphic $(m,0)$-form is called an {\it almost Calabi--Yau} manifold. 
Hence, in our case, $(X(B),\omega,\Omega)$ is always an almost Calabi--Yau manifold. 

From $(B,\mathcal{D}, K)$, we can construct another almost Calabi--Yau manifold by using the cotangent bundle of $B$ as follows. 
As above, the dual lattice bundle $\Lambda^{*}\subset T^{*}B$ is defied by 
\[\Lambda^{*}|_{U_{\lambda}}:=\mathbb{Z} dx_{1}+\cdots+\mathbb{Z} dx_{m}\]
on each $U_{\lambda}$. 
Then, a symplectic manifold $W(B)$ is defined by
\[W(B):=T^{*}B/\Lambda^{*}\]
with the canonical symplectic form 
\[\tilde{\omega}:=\sum_{i}dx_{i}\wedge d\tilde{y}_{i}, \]
where $\tilde{y}_{i}$ is the coefficient of $d x_{i}$ of an element in $T^{*}_{x}B$. 
We denote the torus fibration by $\tilde{f}:W(B)\to B$. 
A complex structure on $W(B)$ is defined by coordinates 
$\tilde{z}_{i}:=\tilde{x}_{i}+\sqrt{-1}\tilde{y}_{i}$ on each $U_{\lambda}$, 
where $\tilde{x}_{i}=\tilde{x}_{i}(x):=\partial K_{\lambda}/\partial x_{i}$. 
A nonvanishing holomorphic $(m,0)$-form is also defined by 
\[\tilde{\Omega}:=d\tilde{z}_{1}\wedge\dots \wedge d\tilde{z}_{m}.\]
Then, $(W(B),\tilde{\omega},\tilde{\Omega})$ is also an almost Calabi--Yau manifold. 
We call $(X(B),\omega,\Omega)$ and $(W(B),\tilde{\omega},\tilde{\Omega})$ a {\it mirror pair} 
associated with $(B,\mathcal{D}, K)$. 
One side is also called the {\it dual torus fibration} of the other side. 
\section{A fiber bundle over $B$}\label{bundle}
Let $V$ be a $k$-dimensional affine linear subspace in $\mathbb{R}^{m}$ ($0\leq k \leq m$). 
In this paper, we say that $V$ is {\it rational} if there exist 
$\xi_{1},\dots,\xi_{k}$, $\zeta_{k+1},\dots,\zeta_{m}$ in $\mathbb{Z}^{m}$ and $a\in \mathbb{R}^{m}$ such that 
\[
\det(\xi_{1},\dots,\xi_{k}, \zeta_{k+1},\dots,\zeta_{m})=\pm 1\quad\text{and}\quad 
V=\mathbb{R}\xi_{1}+\dots+\mathbb{R}\xi_{k}+a. 
\]
Put $\overline{V}:=\mathbb{R}\xi_{1}+\dots+\mathbb{R}\xi_{k}$ and call it the linear part of $V$. 
We denote the set of all $k$-dimensional rational affine linear subspaces in $\mathbb{R}^{m}$ by 
$F(k,m)$. 
For $(b,A)\in \mathbb{R}^{m}\rtimes GL(\mathbb{Z}^{m})$, we define its action on $F(k,m)$ by 
\begin{equation}\label{action}
(b,A)\cdot V:=AV+b.
\end{equation}
One can easily check that $AV+b$ is also in $F(k,m)$. 

Let $(B,\mathcal{D}, K)$ be an $m$-dimensional tropical manifold. 
On $U_{\lambda}\cap U_{\mu}$, put 
\begin{equation}\label{Ab}
A(\lambda,\mu)_{i}^{j}:=\frac{\partial (\psi_{\mu}\circ \psi_{\lambda}^{-1})^{i}}{\partial x_{j}^{\lambda}}\quad\text{and}\quad
b(\lambda,\mu)^{j}:=\frac{\partial (K_{\lambda}-K_{\mu})}{\partial x_{j}^{\lambda}},
\end{equation}
where $(x_{i}^{\lambda})_{i=1}^{m}$ are tropical coordinates on $U_{\lambda}$. 
Then, by definition, $A(\lambda,\mu)_{i}^{j}$ and $b(\lambda,\mu)^{j}$ are constants and 
$A(\lambda,\mu):=(A(\lambda,\mu)_{i}^{j})\in GL(\mathbb{Z}^{m})$. 
Since 
\[A(\lambda,\mu)_{i}^{j}=\frac{\partial x_{i}^{\mu}}{\partial x_{j}^{\lambda}}=
\sum_{\ell=1}^{m}\frac{\partial x_{i}^{\mu}}{\partial x_{\ell}^{\nu}}\frac{\partial x_{\ell}^{\nu}}{\partial x_{j}^{\lambda}}
=\sum_{\ell=1}^{m} A(\lambda,\nu)_{\ell}^{j}A(\nu,\mu)^{\ell}_{i}\]
on $U_{\lambda}\cap U_{\nu}\cap U_{\mu}$, we have $A(\lambda,\mu)=A(\lambda,\nu)A(\nu,\mu)$. 
Furthermore, since 
\[
\begin{aligned}
b(\lambda,\mu)^{j}=&\frac{\partial (K_{\lambda}-K_{\nu}+K_{\nu}-K_{\mu})}{\partial x_{j}^{\lambda}}\\
=&b(\lambda,\nu)^{j}+\sum_{\ell=1}^{m}\frac{\partial (K_{\nu}-K_{\mu})}{\partial x_{\ell}^{\nu}}\frac{\partial x_{\ell}^{\nu}}{\partial x_{j}^{\lambda}}
=b(\lambda,\nu)^{j}+\sum_{\ell=1}^{m} A(\lambda,\nu)_{\ell}^{j}b(\nu,\mu)^{\ell}, 
\end{aligned}
\]
we have $b(\lambda,\mu)=b(\lambda,\nu)+A(\lambda,\nu)b(\nu,\mu)$. 
These imply that 
\begin{equation}\label{cocycle}
(b(\lambda,\mu),A(\lambda,\mu))=(b(\lambda,\nu),A(\lambda,\nu))\cdot (b(\nu,\mu),A(\nu,\mu))
\end{equation}
in $\mathbb{R}^{m}\rtimes GL(\mathbb{Z}^{m})$. 
By this observation, we obtain an $F(k,m)$-bundle over $B$. 
More precisely, let $U_{\lambda}\times F(k,m)$ be the trivial $F(k,m)$-bundle over $U_{\lambda}$. 
Define transition functions $\Theta(\lambda,\mu):U_{\lambda}\cap U_{\mu}\to \mathop{\mathrm{Bij}}(F(k,m))$ by 
\[\Theta(\lambda,\mu)(x):=(b(\lambda,\mu),A(\lambda,\mu)), \]
where $\mathop{\mathrm{Bij}}(F(k,m))$ is the set of all bijections from $F(k,m)$ to itself and 
we identify elements in $\mathbb{R}^{m}\rtimes GL(\mathbb{Z}^{m})$ with elements in $\mathop{\mathrm{Bij}}(F(k,m))$ 
by the action \eqref{action}. 
The relation \eqref{cocycle} says that $\Theta(\lambda,\mu)$ satisfy the cocycle condition, 
and hence we get a $F(k,m)$-bundle over $B$ by gluing each trivialization with the transition functions. 
We denote this $F(k,m)$-bundle over $B$ by $\mathcal{F}(k,m)$ and the projection by $\pi:\mathcal{F}(k,m)\to B$. 
We also denote the local trivialization by $\Psi_{\lambda}:\pi^{-1}(U_{\lambda})\to U_{\lambda}\times F(k,m)$ and 
the set of all sections of $\mathcal{F}(k,m)$ over $B$ by $\Gamma(B,\mathcal{F}(k,m))$. 

\begin{definition}
We say that a map $V_{\lambda}:U_{\lambda}\to F(k,m)$ is {\it constant} if there exist 
$\xi_{1},\dots,\xi_{k}$, $\zeta_{k+1},\dots,\zeta_{m}$ in $\mathbb{Z}^{m}$ and $a\in \mathbb{R}^{m}$ which satisfy 
\begin{equation}\label{constcondi}
\det(\xi_{1},\dots,\xi_{k}, \zeta_{k+1},\dots,\zeta_{m})=\pm 1\quad\text{and}\quad 
V_{\lambda}(x)=\mathbb{R}\xi_{1}+\dots+\mathbb{R}\xi_{k}+a
\end{equation}
for all $x\in U_{\lambda}$, that is, $a$ does not depend on $x\in U_{\lambda}$. 
For a global section $V\in \Gamma(B,\mathcal{F}(k,m))$, define a local section $V_{\lambda}:U_{\lambda}\to F(k,m)$ 
by 
\begin{equation}\label{defofVlam}
(x,V_{\lambda}(x))=\Psi_{\lambda}\circ V(x). 
\end{equation}
Then, a section $V\in \Gamma(B,\mathcal{F}(k,m))$ is {\it constant} 
if $V_{\lambda}$ is constant for all $\lambda\in\Lambda$. 
\end{definition}

\begin{remark}
When $k=m$, then $F(m,m)=\{\,\mathbb{R}^{m}\,\}$, a one-point set. 
Hence, $\mathcal{F}(m,m)$ is a trivial one-point set bundle over $M$ and 
the constant section trivially exists. 
However, for $k<m$, the existence of constant sections of $\mathcal{F}(k,m)\to B$ is not assured. 
\end{remark}
\section{Construction Data}\label{cdata}
Let $(B,\mathcal{D}, K)$ be an $m$-dimensional tropical manifold. 
Fix an integer $k$ with $0\leq k\leq m$ and let $\mathcal{F}(k,m)$ 
be the $F(k,m)$-bundle over $B$ defined in Section \ref{bundle}. 
Assume that there exist at least one constant section of $\mathcal{F}(k,m)$ over $B$. 
Fix a constant section $V\in \Gamma(B,\mathcal{F}(k,m))$. 
From $V$, we construct the associated submanifold $B(V)$ in $B$ with dimension $k$ as follows.
Define a collection of maps $\{\,\mu_{\lambda}:U_{\lambda}\to \mathbb{R}^{m}\,\}_{\lambda \in \Lambda}$ by
\begin{equation}\label{moment}
\mu_{\lambda}(x):=\tilde{x}_{1}(x)e_{1}+\cdots+\tilde{x}_{m}(x)e_{m},
\end{equation}
where $(\tilde{x}_{i})_{i=1}^{m}$ is the Legendre transform of the tropical coordinates on $U_{\lambda}$. 
This is an analogue of a (locally defined) moment map of a toric K\"ahler manifold which is restricted to its real form, see \eqref{localmoment}. 
In this case, the real form is $B$. 
We observe the relation $\mu_{\lambda}$ and $\mu_{\nu}$ on the intersection $U_{\lambda}\cap U_{\nu}$. 
By the formula \eqref{change} and definitions \eqref{Ab}, we have
\[
\tilde{x}_{i}^{\lambda}=\sum_{j=1}^{m}A(\lambda,\nu)_{j}^{i}\tilde{x}_{j}^{\nu}+b(\lambda,\nu)^{i}. 
\]
This implies that 
\begin{equation}\label{momo}
\mu_{\lambda}(x)=A(\lambda,\nu)\mu_{\nu}+b(\lambda,\nu). 
\end{equation}
Define $U_{\lambda}(V):=\{\,x\in U_{\lambda} \mid \mu_{\lambda}(x)\in V_{\lambda}(x) \,\}$. 

\begin{proposition}\label{localsub}
$U_{\lambda}(V)$ is a $k$-dimensional submanifold in $U_{\lambda}$. 
\end{proposition}
\begin{proof}
Since $V$ is a constant section, 
for each $\lambda$ there exist 
$\xi_{1},\dots,\xi_{k}$, $\zeta_{k+1},\dots,\zeta_{m}$ in $\mathbb{Z}^{m}$ and $a\in \mathbb{R}^{m}$ which satisfy \eqref{constcondi}. 
By using the standard basis $e_{1},\dots,e_{m}$, we get an isomorphism 
$\mathbb{R}^{m}\cong (\mathbb{R}^{m})^{*}$, and $\xi^{1},\dots,\xi^{k},\zeta^{k+1},\dots,\zeta^{m}$ denote the dual basis of 
the basis $\xi_{1},\dots,\xi_{k}$, $\zeta_{k+1},\dots,\zeta_{m}$. 
Then, we have $V_{\lambda}(x)=\{\,\tilde{x}\in\mathbb{R}^{m}\mid \langle \tilde{x},\zeta^{j}\rangle = \langle a,\zeta^{j}\rangle\, ,\, j=k+1,\dots,m \,\}$ and 
also have
\[U_{\lambda}(V)=\{\,x\in U_{\lambda}\mid \langle \mu_{\lambda}(x),\zeta^{j}\rangle = \langle a,\zeta^{j}\rangle\, ,\, j=k+1,\dots,m \,\}.\]
For $j=k+1,\dots,m$, define functions $f_{j}$  on $U_{\lambda}$ by 
$f_{j}(x):=\langle \mu_{\lambda}(x),\zeta^{j}\rangle - \langle a,\zeta^{j}\rangle$. 
Then, $U_{\lambda}(V)$ is the intersection of zero sets of $f_{j}$. 
Note that 
\begin{equation}\label{localmoment}
\begin{aligned}
df_{j}=&d\langle \mu_{\lambda},\zeta^{j}\rangle
=\langle e_{1},\zeta^{j}\rangle d\tilde{x}_{1}+\cdots+\langle e_{n},\zeta^{j}\rangle d\tilde{x}_{m}\\
=&\langle e_{1},\zeta^{j}\rangle \sum_{\ell=1}^{m}\frac{\partial^2 K_{\lambda}}{\partial x_{1}\partial x_{\ell}}dx_{\ell}+\cdots+\langle e_{n},\zeta^{j}\rangle 
\sum_{\ell=1}^{m}\frac{\partial^2 K_{\lambda}}{\partial x_{m}\partial x_{\ell}}dx_{\ell}
=-\omega\left(\zeta^{j\sharp}, \,\cdot\, \right),
\end{aligned}
\end{equation}
where $\omega$ is the symplectic form defined by \eqref{symp} and $\zeta^{j\sharp}$ is defined by 
\[
\zeta^{j\sharp}:=\langle e_{1},\zeta^{j}\rangle \frac{\partial}{\partial y_{1}}+\cdots+\langle e_{m},\zeta^{j}\rangle \frac{\partial}{\partial y_{m}}. 
\]
Since $\omega$ is non-degenerate and $\zeta^{j\sharp}$ ($j=k+1,\dots,m$) are linearly independent, 
we see that $df_{j}$ ($j=k+1,\dots,m$) are linearly independent. 
Thus, by the implicit function theorem, $U_{\lambda}(V)$ is a $k$-dimensional submanifold in $U_{\lambda}$. 
\end{proof}

Next, we will see that each $U_{\lambda}(V)$ coincides in the intersection of $U_{\lambda}$ and $U_{\mu}$. 

\begin{proposition}\label{gluesub}
$U_{\lambda}(V)=U_{\nu}(V)$ in $U_{\nu}\cap U_{\lambda}$. 
\end{proposition}
\begin{proof}
Assume that there exists a point $x\in U_{\lambda}(V)\cap U_{\mu}$. 
By \eqref{momo}, we have 
\[\mu_{\lambda}(x)=A(\lambda,\nu)\mu_{\nu}(x)+b(\lambda,\nu)\in V_{\lambda}(x).\]
Recall that $V_{\lambda}(x)$ is defined by \eqref{defofVlam}. 
This together with the definition of the bundle $\mathcal{F}(k,m)$ implies that 
\[
V_{\lambda}(x)=A(\lambda,\nu)V_{\nu}(x)+b(\lambda,\nu). 
\]
Thus, we proved that $\mu_{\nu}(x)\in V_{\nu}(x)$ and this yields $U_{\lambda}(V)\cap U_{\nu}\subset U_{\nu}(V)\cap U_{\lambda}$. 
Interchanging $\lambda$ and $\mu$, we get the counter inclusion, and the proof is complete. 
\end{proof}

By Proposition \ref{localsub} and \ref{gluesub}, the following definition makes sense. 
\begin{definition}
For $V\in \Gamma(B,\mathcal{F}(k,m))$, a submanifold $B(V)$ in $B$ is defined by
\[B(V):=\bigcup_{\lambda\in\Lambda}U_{\lambda}(V). \]
\end{definition}
Note that $\mathrm{dim}_{\mathbb{R}}B(V)=k$. Roughly speaking, $B(V)$ is a set what one wants to write as $\{\,x\in B \mid \mu(x)\in V(x) \,\}$. 
However, in this setting, we do not have globally defined $\mu$. 
Hence, we defined $U_{\lambda}(V)$ locally and took the union of these. 

We denote the inclusion map by $\iota:B(V) \hookrightarrow B$. 
Then, we get the pull-back bundle $\iota^{*}(TB)$ over $B(V)$. 
We denote the set of all smooth sections of $\iota^{*}(TB)$ over $B(V)$ by $\Gamma(B(V),\iota^{*}(TB))$. 
In summary, we need the following data. 

\begin{definition}
Let $k$ be an integer with $0\leq k\leq m$, 
$V\in \Gamma(B,\mathcal{F}(k,m))$ be a constant section and 
$Y\in \Gamma(B(V),\iota^{*}(TB))$ be a smooth section. 
We call the triplet $(k,V,Y)$ {\it construction data}. 
\end{definition}

Let $(k,V,Y)$ be construction data. 
Here, we prepare some notations to write propositions and proofs simply. 
Assume that $p\in B(V)$ is given. 
Take a tropical chart $(U_{\lambda}, x_{1},\dots,x_{m})$ around $p$. 
Then, for this $\lambda$, there exist 
$\xi_{1},\dots,\xi_{k}$, $\zeta_{k+1},\dots,\zeta_{m}$ in $\mathbb{Z}^{m}$ and $a\in \mathbb{R}^{m}$ which satisfy \eqref{constcondi}. 
Thus, by the definition of $B(V)$, the assignment 
\begin{equation}\label{natpara}
u=(u^{1},\dots,u^{k})\mapsto \mu_{\lambda}^{-1}(u^{1}\xi_{1}+\dots+u^{k}\xi_{k}+a)\in U_{\lambda}(V)
\end{equation}
gives a local parametrization around $p$ when the domain of $u$ is appropriately chosen. 
We call this a {\it natural parametrization} of $B(V)$. 
Define locally defined $\mathbb{C}$-valued functions by 
\[c^{i}_{j}(u):=\frac{\partial x^{i}}{\partial u^{j}}+\sqrt{-1}\frac{\partial Y^{i}}{\partial u^{j}}, \]
where $Y^{i}(u)$ is the coefficient of $\partial/\partial x^{i}$ of $Y(u)$. 
Put 
\begin{equation}\label{WZ}
\mathcal{W}(u):=
\left[
\begin{array}{ccc}
c^{1}_{1}(u)     &   \cdots   &  c^{1}_{k}(u) \\
c^{2}_{1}(u)     &   \cdots   &  c^{2}_{k}(u)  \\
\vdots              &   \ddots   &  \vdots       \\
c^{m}_{1}(u)    &   \cdots    &  c^{m}_{k}(u)
\end{array}
\right]
,
\mathcal{Z}:=
\left[
\begin{array}{ccc}
  \langle e_{1},\zeta^{k+1}\rangle  &  \cdots   &  \langle e_{1},\zeta^{m}\rangle\\
  \langle e_{2},\zeta^{k+1}\rangle  &  \cdots   &  \langle e_{2},\zeta^{m}\rangle\\
  \vdots         &  \ddots   &  \vdots      \\
 \langle e_{m},\zeta^{k+1}\rangle  &  \cdots   &  \langle e_{m},\zeta^{m}\rangle
\end{array}
\right],
\end{equation}
and denote the augmentation of $\mathcal{W}(u)$ and $\mathcal{Z}$, an $m\times m$ matrix, by $[\mathcal{W}(u)|\mathcal{Z}]$. 
In Proposition \ref{sLagpt} and \ref{dHYMprop}, we will see the same equation; 
\begin{equation}\label{MA}
\arg\det\left[\mathcal{W}(u)|\mathcal{Z} \right]=\theta_{0}. 
\end{equation}
When $Y=\nabla f$ for some function $f$ on $B(V)$, the equation \eqref{MA} is 
considered as a certain kind of  Monge--Amp\`ere equation. 
\section{Construction of a Lagrangian submanifold in $X(B)$}\label{LAG}
Let $(B,\mathcal{D}, K)$ be an $m$-dimensional tropical  manifold. 
Fix construction data $(k,V,Y)$. 
Let $(x_{i}^{\lambda})_{i=1}^{m}$ be tropical coordinates on $U_{\lambda}$. 
For each $x\in U_{\lambda}$, an isomorphism $\Theta_{\lambda}:(\mathbb{R}^{m})^{*}\to T_{x}B$ is defined by 
\[\Theta_{\lambda}\left(\sum_{i=1}^{m}y_{i}e^{i}\right):=\sum_{i=1}^{m}y_{i}\frac{\partial}{\partial x_{i}^{\lambda}},\]
where $(e^{i})_{i=1}^{m}$ is the standard dual basis of $(\mathbb{R}^{m})^{*}$. 
Let $V_{\lambda}(x)$ be the local expression of $V(x)$ defined by \eqref{defofVlam}. 
Define the conormal space of the linear part of $V_{\lambda}(x)$ by
\[\overline{V}_{\lambda}^{\bot}(x):=\{\,w\in (\mathbb{R}^{m})^{*}\mid \langle v,w\rangle=0\,\,\text{for all}\,\,v\in \overline{V}_{\lambda}(x) \,\}.\]
Since $V$ is a constant section, for each $\lambda$ 
there exist 
$\xi_{1},\dots,\xi_{k}$, $\zeta_{k+1},\dots,\zeta_{m}$ in $\mathbb{Z}^{m}$ and $a\in \mathbb{R}^{m}$ which satisfy \eqref{constcondi}. 
Denote the dual basis of $\xi_{i}$, $\zeta_{j}$ by $\xi^{i}$, $\zeta^{j}$. 
Then, it is clear that these dual basis is also unimodular and 
$\overline{V}_{\lambda}^{\bot}(x)$ is spanned by $\{\, \zeta^{j}\,\}_{j=k+1}^{m}$. 
Hence, $\overline{V}_{\lambda}^{\bot}(x)/(\mathbb{Z}^{m})^{*}$ is isomorphic to an $(m-k)$-dimensional torus $T^{m-k}$. 
By using the map $\Theta_{\lambda}:(\mathbb{R}^{m})^{*}\to T_{x}B$, 
an $(m-k)$-dimensional subtorus $T(V^{\bot})_{x}$ in the torus fiber $f^{-1}(x)=T_{x}B/\Lambda_{x}$ is defined by 
\begin{equation}\label{TVnorm}
T(V^{\bot})_{x}:=\Theta_{\lambda}(\overline{V}_{\lambda}^{\bot}(x))/\Lambda_{x}. 
\end{equation}

\begin{lemma}\label{normal}
$\Theta_{\lambda}(\overline{V}_{\lambda}^{\bot}(x))=T_{x}^{\bot}B(V)$, where $\bot$ in the right hand side means the orthogonal complement of $T_{x}B(V)$ in $T_{x}B$ 
with respect to the Riemannian metric $g$ on $B$. 
\end{lemma}
\begin{proof}
We use notations and formulas appeared in the proof of Proposition \ref{localsub}. 
It follows that $df_{i}(v)=0$ for any $v\in T_{x}B(V)$ from $f_{j}=0$ on $B(V)$. 
This yields 
\begin{equation}\label{zero}
0=-\omega\left(\zeta^{j\sharp}, v \right)=g\left(-J\zeta^{j\sharp}, v \right). 
\end{equation}
Note that
\[-J\zeta^{j\sharp}=\langle e_{1},\zeta^{j}\rangle \frac{\partial}{\partial x_{1}}+\cdots+\langle e_{m},\zeta^{j}\rangle \frac{\partial}{\partial x_{m}}=\Theta_{\lambda}(\zeta^{j}).\]
This implies that $\Theta_{\lambda}(\zeta^{j})$ ($j=k+1,\dots,m$) is a basis of $T_{x}^{\bot}B(V)$. 
On the other hand, $\Theta_{\lambda}(\zeta^{j})$ ($j=k+1,\dots,m$) is nothing but a basis of $\Theta_{\lambda}(\overline{V}_{\lambda}^{\bot}(x))$. 
Then, the proof is complete. 
\end{proof}

From Lemma \ref{normal}, it follows that the definition of $T(V^{\bot})_{x}$ does not depend on the choice of the tropical chart. 
Here, recall that a smooth section $Y\in \Gamma(B(V),\iota^{*}(TB))$ is chosen as a part of the construction data. 
Hence, for each $x\in B(V)$, we can translate the subtorus $T(V^{\bot})_{x}$ in $f^{-1}(x)=T_{x}B/\Lambda_{x}$ by adding $Y(x)$, 
and we denote it by
\[T(V^{\bot},Y)_{x}:=T(V^{\bot})_{x}+Y(x).\]
Note that, in general, this torus is not a subgroup of $f^{-1}(x)$ since it might not pass through the origin. 
By summing up all $T(V^{\bot},Y)_{x}$ over $B(V)$, we get a $T^{m-k}$-bundle over $B(V)$, and we denote its total space by
\[L(V,Y):=\bigcup_{x\in B(V)}T(V^{\bot},Y)_{x}.\]
The dimension of the fiber is $m-k$ and the one of the base is $k$. 
Hence, $L(V,Y)$ is an $m$-dimensional submanifold in $X(B)$, see Figure \ref{fig1}. 

\begin{figure}[h]
\begin{center}
\includegraphics[bb=160 541 435 672, clip]{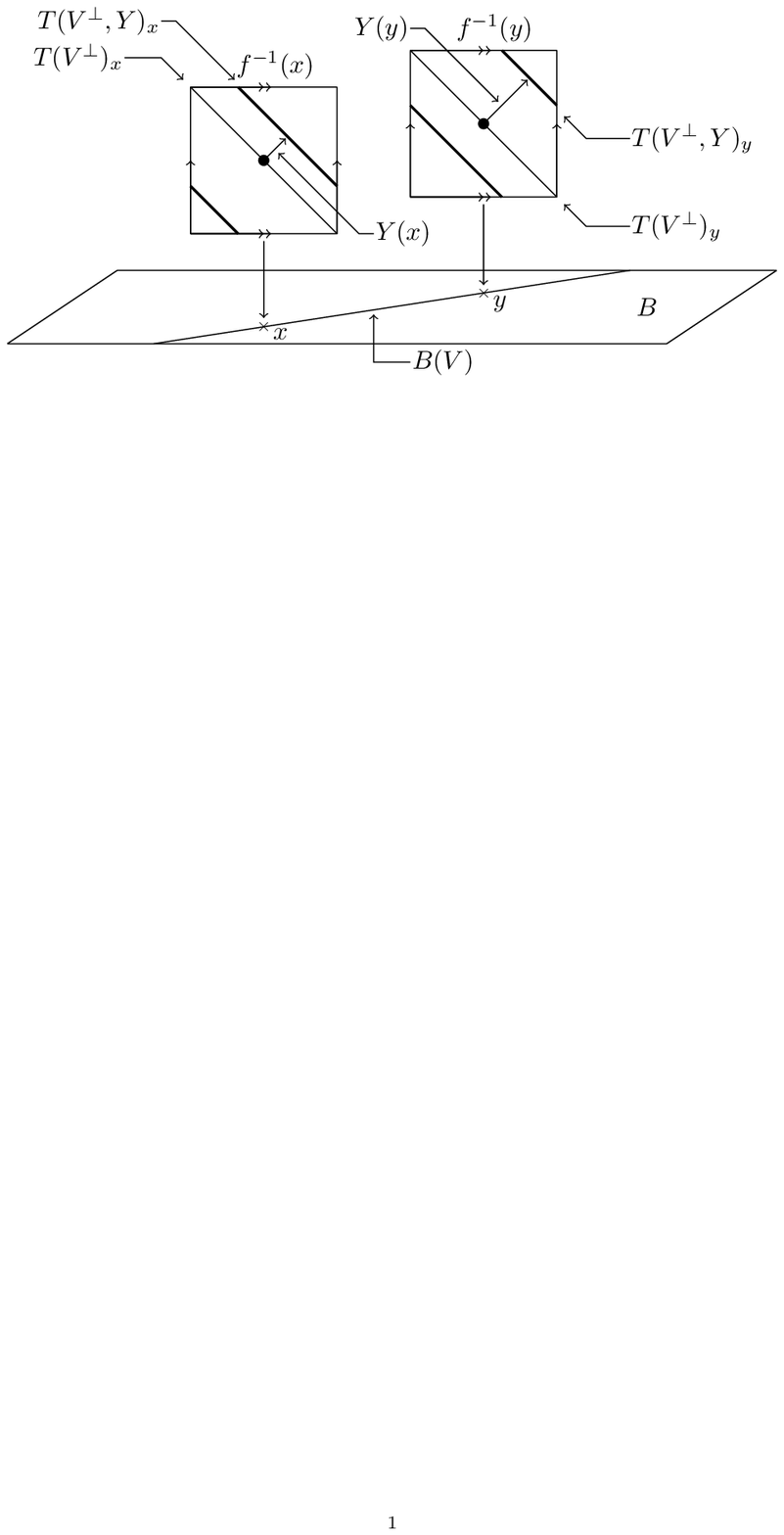}
\end{center}
\caption{$L(V,Y)$}
\label{fig1}
\end{figure}

\begin{proposition}\label{normalprop}
$T(V^{\bot},Y)_{x}=T(V^{\bot})_{x}$ if and only if $Y(x)\in T_{x}^{\bot}B(V)$. 
\end{proposition}
\begin{proof}
This follows immediately from the definition \eqref{TVnorm} and Lemma \ref{normal}. 
\end{proof}

By Proposition \ref{normalprop}, we can assume that 
$Y\in \Gamma(B(V),\iota^{*}(TB))$ satisfies $Y(x)\in T_{x}B(V)$ without loss of generality,
that is, $Y$ is a vector field on $B(V)$. 

\begin{proposition}\label{Lagpt}
$L(V,Y)$ is a Lagrangian submanifold in $(X(B),\omega)$ if and only if $Y$ is locally gradient, that is, 
$Y=\nabla f$ for some locally defined smooth function $f$ on $B(V)$. 
\end{proposition}
\begin{proof}
Define a 1-form $\eta$ on $B(V)$ by $\eta(\,\cdot\,):=\omega(JY,\,\cdot\,)=-g(Y,\,\cdot\,)$. 
Since $Y$ is a tangent vector field, 
the condition such that $Y$ is locally gradient is equivalent to $d\eta=0$. 
This equivalence will be used later. 
Let $u=(u^{1}, \dots,u^{k})\mapsto (x^{i}(u))_{i=1}^{m}$ be a natural parametrization in $U_{\lambda}(V)$ defined by \eqref{natpara}. 
Let $Y^{i}(u)$ be the coefficient of $\partial/\partial x^{i}$ of $Y(u)$. 
Let $t=(t^{k+1},\dots,t^{m})$ be a $(m-k)$-dimensional parameter. 
Then, a local parametrization of $L(V,Y)$ is defined by 
\begin{equation}\label{paraut}
(u,t)\mapsto \left(x(u),\sum_{i=k+1}^{m}t^{i}\Theta_{\lambda}(\zeta^{i})+Y(u)\right)\in L(V,Y)\subset TU_{\lambda}/(\Lambda|_{U_{\lambda}}). 
\end{equation}
Hence, the tangent vectors associated with this parametrization are given by 
\[
W_{j}:=\frac{\partial}{\partial u^{j}}+\sum_{\ell=1}^{m}\frac{\partial Y^{\ell}(u)}{\partial u^{j}}\frac{\partial}{\partial y^{\ell}}, \quad
Z_{i}:=\sum_{\ell=1}^{m}\langle e_{\ell},\zeta^{i}\rangle \frac{\partial}{\partial y_{\ell}}=\zeta^{i\sharp}. 
\]
First, we have $\omega(Z_{i},Z_{i'})=0$ since $Z_{i}$ is tangent to a torus fiber and each torus fiber is Lagrangian. 
Next, we will show that $\omega(W_{j},Z_{i})=0$. 
Since $\partial/\partial y^{\ell}$ is tangent to a torus fiber and each torus fiber is Lagrangian, we have
\[\omega(W_{j},Z_{i})=\omega\left(\frac{\partial}{\partial u^{j}}, \zeta^{i\sharp}\right). \]
Then, by \eqref{zero}, we have $\omega(W_{j},Z_{i})=0$. 
Note that the above argument does not depend on whether $Y$ is locally gradient or not. 
Finally, we compute $\omega(W_{j},W_{j'})$. 
Note that 
\begin{equation}\label{Wjj}
\begin{aligned}
\omega(W_{j},W_{j'})=&\omega\left(\frac{\partial}{\partial u^{j}},\sum_{\ell'=1}^{m}\frac{\partial Y^{\ell'}}{\partial u^{j'}}\frac{\partial}{\partial y^{\ell'}} \right)
 +\omega\left(\sum_{\ell=1}^{m}\frac{\partial Y^{\ell}}{\partial u^{j}}\frac{\partial}{\partial y^{\ell}} ,\frac{\partial}{\partial u^{j'}}\right)\\
=&\sum_{\ell',i=1}^{m}\frac{\partial^2 K_{\lambda}}{\partial x^{i}\partial x^{\ell'}}\frac{\partial Y^{\ell'}}{\partial u^{j'}}\frac{\partial x^{i}}{\partial u^{j}}
-\sum_{\ell,i=1}^{m}\frac{\partial^2 K_{\lambda}}{\partial x^{\ell}\partial x^{i}}\frac{\partial Y^{\ell}}{\partial u^{j}}\frac{\partial x^{i}}{\partial u^{j'}} . 
\end{aligned}
\end{equation}
 On the other hand, the coefficient of $du^{j}$ of $\eta$ is given by
 \begin{equation}\label{etaj}
 \begin{aligned}
\eta_{j}=&\omega\left(Y,\frac{\partial}{\partial u^{j}}\right)
=\omega\left(\sum_{i=1}^{m}Y^{i}\frac{\partial}{\partial y^{i}},\sum_{\ell=1}^{m}\frac{\partial x^{\ell}}{\partial u^{j}}\frac{\partial}{\partial x^{\ell}}\right)
=-\sum_{i,\ell=1}^{m}\frac{\partial^2 K_{\lambda}}{\partial x^{i}\partial x^{\ell}}Y^{i}\frac{\partial x^{\ell}}{\partial u^{j}}. 
 \end{aligned}
 \end{equation}
 Differentiating \eqref{etaj} with respect to $u^{j'}$ gives 
\[
\begin{aligned}
\frac{\partial \eta_{j}}{\partial u^{j'}}=&-\sum_{\ell,\ell',i=1}^{m}\frac{\partial^3 K_{\lambda}}{\partial x^{\ell'}\partial x^{i}\partial x^{\ell}}
\frac{\partial x^{\ell'}}{\partial u^{j'}}Y^{i}\frac{\partial x^{\ell}}{\partial u^{j}}
-\sum_{\ell,i=1}^{m}\frac{\partial^2 K_{\lambda}}{\partial x^{i}\partial x^{\ell}}\frac{\partial Y^{i}}{\partial u^{j'}}\frac{\partial x^{\ell}}{\partial u^{j}}\\
&\quad \quad -\sum_{\ell,i=1}^{m}\frac{\partial^2 K_{\lambda}}{\partial x^{i}\partial x^{\ell}}Y^{i}\frac{\partial^2 x^{\ell}}{\partial u^{j'}\partial u^{j}}.
\end{aligned}
\]
Since the first and third term are symmetric with respect to $j,j'$, the condition such that 
\[\frac{\partial \eta_{j}}{\partial u^{j'}}=\frac{\partial \eta_{j'}}{\partial u^{j}}\]
is equivalent to 
\[\sum_{\ell,i=1}^{m}\frac{\partial^2 K_{\lambda}}{\partial x^{i}\partial x^{\ell}}\frac{\partial Y^{i}}{\partial u^{j'}}\frac{\partial x^{\ell}}{\partial u^{j}}
=\sum_{\ell,i=1}^{m}\frac{\partial^2 K_{\lambda}}{\partial x^{i}\partial x^{\ell}}\frac{\partial Y^{i}}{\partial u^{j}}\frac{\partial x^{\ell}}{\partial u^{j'}}. \]
The former means $d\eta=0$, and the latter means $\omega(W_{j},W_{j'})=0$ for all $j,j'$ by \eqref{Wjj}. 
Thus, the proof is complete. 
\end{proof}

Next, we consider the special Lagrangian condition for $L(V,Y)$. 

\begin{proposition}\label{sLagpt}
Assume that $L(V,Y)$ is a Lagrangian submanifold. 
Then, it is a special Lagrangian submanifold with phase $e^{\sqrt{-1}\theta_{0}}$ in $(X(B),\omega,\Omega)$ 
if and only if $V$ and $Y$ satisfy 
\[\arg\det[\mathcal{W}(u)|\mathcal{Z}]=\theta_{0}\]
for each natural parametrization $u$, see \eqref{natpara} and \eqref{WZ}. 
\end{proposition}
\begin{proof}
Let $u=(u^{i})_{i=1}^{k}\mapsto (x^{i}(u))_{i=1}^{m}$ be a natural parameterization of $B(V)$ in $U_{\lambda}$ and 
let $(u,t)$ be the local coordinates given by \eqref{paraut}. 
Then, we have
\[
\begin{aligned}
dz^{i}=&dx^{i}+\sqrt{-1}dy^{i}=\sum_{j=1}^{k}\left(\frac{\partial x^{i}}{\partial u^{j}}+\sqrt{-1}\frac{\partial Y^{i}}{\partial u^{j}}\right)du^{j}
+\sqrt{-1}\sum_{j=k+1}^{m}\langle e_{i},\zeta^{j}\rangle dt^{j}\\
=&\sum_{j=1}^{k}c^{i}_{j}(u)du^{j}+\sqrt{-1}\sum_{j=k+1}^{m}\langle e_{i},\zeta^{j}\rangle dt^{j}
\end{aligned}
\]
as a 1-form on $L(V,Y)$. 
Then, it is clear that the coefficient of 
\[du^{1}\wedge\dots\wedge du^{k}\wedge dt^{k+1}\wedge\dots\wedge dt^{m}\]
of $\Omega=dz^{1}\wedge\dots\wedge dz^{m}$ restricted to $L(V,Y)$ is $\det[\mathcal{W}(u)|\mathcal{Z}]$. 
Then, the proof is complete. 
\end{proof}
\section{Construction of a complex submanifold with a connection in $W(B)$}\label{ccc}
Let $(B,\mathcal{D}, K)$ be an $m$-dimensional tropical manifold. 
Fix construction data $(k,V,Y)$. 
Let $(x_{i}^{\lambda})_{i=1}^{m}$ be tropical coordinates on $U_{\lambda}$. 
For each $x\in U_{\lambda}$, define an isomorphism $\tilde{\Theta}_{\lambda}:\mathbb{R}^{m}\to T_{x}^{*}B$ by 
\[\tilde{\Theta}_{\lambda}\left(\sum_{i=1}^{m}y_{i}e_{i}\right):=\sum_{i=1}^{m}y_{i}dx_{i}^{\lambda},\]
where $(e_{i})_{i=1}^{m}$ is the standard basis of $\mathbb{R}^{m}$. 
Let $V_{\lambda}(x)$ be the local expression of $V(x)$ defined by \eqref{defofVlam}. 
Recall that its linear part is denoted by $\overline{V}_{\lambda}(x)$. 
Since $V$ is a constant section, for each $\lambda$ there exist 
$\xi_{1},\dots,\xi_{k}$, $\zeta_{k+1},\dots,\zeta_{m}$ in $\mathbb{Z}^{m}$ and $a\in \mathbb{R}^{m}$ which satisfy \eqref{constcondi}. 
Since this basis is unimodular, $\overline{V}_{\lambda}(x)/\mathbb{Z}^{m}$ is isomorphic to a $k$-dimensional torus $T^{k}$. 
By using the map $\tilde{\Theta}_{\lambda}:\mathbb{R}^{m}\to T_{x}^{*}B$, a $k$-dimensional subtorus $T(V)_{x}$ in the torus fiber $\tilde{f}^{-1}(x)=T_{x}^{*}B/\Lambda^{*}_{x}$ is defined by 
\[
T(V)_{x}:=\tilde{\Theta}_{\lambda}(\overline{V}_{\lambda}(x))/\Lambda_{x}^{*}. 
\]
Form Lemma \ref{normal}, it follows that $\Theta_{\lambda}(\overline{V}^{\bot}_{\lambda}(x))=T^{\bot}_{x}B(V)$. 
Thus, $\tilde{\Theta}_{\lambda}(\overline{V}_{\lambda}(x))$ is just the conormal space of $T^{\bot}_{x}B(V)$, 
precisely 
\[\tilde{\Theta}_{\lambda}(\overline{V}_{\lambda}(x))=\{\,w\in T_{x}^{*}B(V)\mid \langle v,w\rangle=0 \,\,\text{for all}\,\, v\in T_{x}^{\bot}B(V)\,\}, \]
and this implies that $T(V)_{x}$ does not depend on the choice of the tropical chart. 
By summing up all $T(V)_{x}$ over $B(V)$, we get a $T^{k}$-bundle over $B(V)$, and we denote its total space by
\[C(V):=\bigcup_{x\in B(V)}T(V)_{x}.\]
The dimension of the fiber is $k$ and the one of the base is $k$. 
Hence, $C(V)$ is a real $2k$-dimensional submanifold in $W(B)$. 

\begin{lemma}\label{lem3}
$C(V)$ is a complex submanifold in $W(B)$. 
\end{lemma}
\begin{proof}
Let $u=(u^{i})_{i=1}^{k}\mapsto (x^{i}(u))_{i=1}^{m}$ be a natural parameterization of $B(V)$ in $U_{\lambda}$. 
We denote the $\tilde{x}_{j}$-coordinates, the Legendre transform of $x_{j}$, at the point $(x^{i}(u))_{i=1}^{m}$ by $\tilde{x}_{j}(u)$. 
Then, \eqref{moment} and \eqref{natpara} yield 
\begin{equation}\label{uzx}
u^{1}\xi_{1}+\dots+u^{k}\xi_{k}+a=\tilde{x}_{1}(u)e_{1}+\dots+\tilde{x}_{m}(u)e_{m}. 
\end{equation}
For $v=(v^{1},\dots,v^{k})$, we assign the point 
\[v^{1}\tilde{\Theta}_{\lambda}(\xi_{1})+\dots+v^{m}\tilde{\Theta}_{\lambda}(\xi_{m})\]
in $T(V)_{x}$. This gives local coordinates on $T(V)_{x}$. 
Also we denote the $\tilde{y}_{j}$-coordinates at this point by $\tilde{y}_{j}(v)$. 
Then, we have 
\begin{equation}\label{vzy}
v^{1}\tilde{\Theta}_{\lambda}(\xi_{1})+\dots+v^{m}\tilde{\Theta}_{\lambda}(\xi_{m})=\tilde{y}_{1}(v)dx^{1}+\dots+\tilde{y}_{m}(v)dx^{m}.
\end{equation}
Thus, $(u,v)\mapsto (\tilde{x}(u),\tilde{y}(v))$ gives local coordinates on $C(V)$. 
Differentiating \eqref{uzx} and \eqref{vzy} gives 
\begin{equation}\label{CR}
\frac{\partial \tilde{x}_{j}}{\partial u^{i}}=\langle \xi_{i},e^{j}\rangle=\frac{\partial \tilde{y}_{j}}{\partial v^{i}}\quad\text{and}\quad
\frac{\partial \tilde{x}_{j}}{\partial v^{i}}=0=-\frac{\partial \tilde{y}_{j}}{\partial u^{i}}. 
\end{equation}
Thus, $(u,v)\mapsto (\tilde{x}(u),\tilde{y}(v))$ is a holomorphic map from some domain in $\mathbb{C}^{k}$ to $W(B)$ 
such that it gives local coordinates on $C(V)$. 
This shows that $C(V)$ is a complex submanifold in $W(B)$ with complex dimension $k$. 
\end{proof}

Let $\underline{\mathbb{C}}:=C(V)\times \mathbb{C}$ be the trivial complex line bundle with the standard Hermitian metric $h:=\langle\,\cdot\, , \, \cdot \,\rangle$. 
Recall that a smooth section $Y\in \Gamma(B(V),\iota^{*}(TB))$ is chosen as a part of construction data. 
Then, a Hermitian connection $D^{Y}$ on $\underline{\mathbb{C}}$ is defined by 
\[D^{Y}:=d+\sqrt{-1}\sum_{j=1}^{m}Y_{j}d\tilde{y}_{j}, \]
where $Y_{j}(x)$ is the coefficient of $\partial/\partial x_{j}$ of $Y(x)$ and $d\tilde{y}_{j}$ is regarded as a 1-form on 
$C(V)$ by pulling it back by the inclusion map $C(V)\hookrightarrow W$. 
Note that this definition does not depend on the choice of tropical coordinates. 
Actually, let $(U_{\lambda},(x^{\lambda}_{i})_{i=1}^{m})$ and $(U_{\nu},(x^{\nu}_{i})_{i=1}^{m})$ be two tropical coordinates which have a nonempty intersection. 
Denote by $Y^{\lambda}_{j}(x)$ and $Y^{\nu}_{j}(x)$ the coefficient of $\partial/\partial x^{\lambda}_{j}$ and $\partial/\partial x^{\nu}_{j}$ of $Y(x)$, respectively. 
Then, the relations
\[Y^{\nu}_{\ell}=\sum_{j=1}^{m}\frac{\partial x^{\nu}_{\ell}}{\partial x^{\lambda}_{j}}Y^{\lambda}_{j}\quad\text{and}\quad
d\tilde{y}^{\nu}_{\ell}=\sum_{j=1}^{m}\frac{\partial x^{\lambda}_{j}}{\partial x^{\nu}_{\ell}}dy^{\lambda}_{j}\]
show that the definition of $D^{Y}$ does not depend on the choice of tropical coordinates. 
The following is a mirror side proposition which corresponds to Proposition \ref{normalprop}. 

\begin{proposition}\label{lem2}
$D^{Y}=d$ if and only if $Y(x)\in T^{\bot}B(V)$ for all $x\in B(V)$. 
\end{proposition}
\begin{proof}
We work on a tropical chart $(U_{\lambda},(x_{i})_{i=1}^{m})$. 
Assume that $D^{Y}=d$. Then, $\sum_{j} Y_{j}(x)d\tilde{y}_{j}=0$ as a 1-form on $C(V)$. 
Let $\{f_{\alpha}\}_{\alpha=1}^{k}$ be a basis of $T_{x}B(V)$. 
Define coefficients $\tilde{f}_{\alpha}^{j}$ and $f_{\alpha}^{j}$ by 
\[f_{\alpha}=\sum_{j=1}^{m}\tilde{f}_{\alpha}^{j}\frac{\partial}{\partial \tilde{x}_{j}}\quad\text{and}\quad
f_{\alpha}=\sum_{j=1}^{m}f_{\alpha}^{j}\frac{\partial}{\partial x_{j}},\]
respectively. 
Then, there is a relation 
\[\tilde{f}_{\alpha}^{j}=\sum_{\ell=1}^{m}\frac{\partial \tilde{x}_{j}}{\partial x_{\ell}}f_{\alpha}^{\ell}
=\sum_{\ell=1}^{m}\frac{\partial^2 K_{\lambda}}{\partial x_{j}\partial x_{\ell}}f^{\ell}_{\alpha}. \]
By the inclusion map $B(V)\hookrightarrow C(V)$, the tangent vector $f_{\alpha}$ is considered as a vector in $T_{x}C(V)$. 
Since $C(V)$ is a complex submanifold, we know that $Jf_{\alpha}$ is also in $T_{x}C(V)$, and it is written as
\[Jf_{\alpha}=\sum_{j=1}^{m}\tilde{f}_{\alpha}^{j}\frac{\partial}{\partial \tilde{y}_{j}}\]
 for $\alpha=1,\dots,k$. 
Then, it follows that 
\begin{equation}\label{eqeq}
\left(\sum_{j=1}^{m}Y_{j}(x)d\tilde{y}_{j}\right)(Jf_{\alpha})=\sum_{j=1}^{m}Y_{j}(x)\tilde{f}_{\alpha}^{j}
=\sum_{j,\ell=1}^{m}\frac{\partial^2 K_{\lambda}}{\partial x_{j}\partial x_{\ell}}Y_{j}(x)f^{\ell}_{\alpha}. 
\end{equation}
By the assumption, the left hand side is $0$. 
This shows that $Y(x)$ is normal to $T_{x}B(V)$ since the Riemannian metric on $B(V)$ is the restriction of 
\[g=\sum_{j,\ell=1}^{m}\frac{\partial^2 K_{\lambda}}{\partial x_{j}\partial x_{\ell}}dx_{\ell}\otimes dx_{j}. \]
The converse also follows from the formula \eqref{eqeq} and the fact that 
the union of $\{f_{\alpha}\}_{\alpha=1}^{k}$ and $\{Jf_{\alpha}\}_{\alpha=1}^{k}$ is a basis of $T_{x}C(V)$. 
\end{proof}

By Proposition \ref{lem2}, we can assume that 
$Y\in \Gamma(B(V),\iota^{*}(TB))$ satisfies $Y(x)\in T_{x}B(V)$ without loss of generality. 
Namely, $Y$ is a vector field on $B(V)$. 
We say that the connection $D^{Y}$ is {\it integrable} if it defines a holomorphic structure on $\underline{\mathbb{C}}$, 
equivalently the $(0,2)$-part of the curvature 2-form $F$ of $D^{Y}$ is zero. 
It is proved in Proposition 2.3 in \cite{Chan} that the existence of a local potential for $Y$ implies the integrability of $D^{Y}$. 
Here, for reader's convenience, we give its equivalence though the proof of the converse is straightforward. 
This is a mirror side proposition which corresponds to Proposition \ref{Lagpt}. 

\begin{proposition}\label{integpt}
$D^{Y}$ is integrable if and only if $Y$ is locally gradient, that is, 
$Y=\nabla f$ for some locally defined smooth function $f$ on $B(V)$. 
\end{proposition}
\begin{proof}
Let $(u,v)\mapsto (\tilde{x}(u),\tilde{y}(v))$ be holomorphic coordinates on $C(V)$ taken in the proof of Lemma \ref{lem3}. 
Then, the curvature 2-form of $D^{Y}$ is given by
\[F=d\left( \sqrt{-1}\sum_{j=1}^{m}Y_{j}d\tilde{y}_{j} \right)
=\sqrt{-1}\sum_{j=1}^{m}\sum_{i=1}^{k}\frac{\partial Y_{j}}{\partial u^{i}}du^{i}\wedge d\tilde{y}_{j}. \]
By \eqref{CR}, we have
$d\tilde{y}_{j}=\sum_{\ell=1}^{k}\langle \xi_{\ell},e^{j}\rangle dv^{\ell}$. 
These give 
\begin{equation}\label{curvature}
F=\sqrt{-1}\sum_{j=1}^{m}\sum_{i,\ell=1}^{k}\langle \xi_{\ell},e^{j}\rangle\frac{\partial Y_{j}}{\partial u^{i}}du^{i}\wedge dv^{\ell}
=\sqrt{-1}\sum_{i,\ell=1}^{k}\frac{\partial \langle \tilde{\Theta}(\xi_{\ell}),Y\rangle}{\partial u^{i}}du^{i}\wedge dv^{\ell}. 
\end{equation}
Put the complex coordinate $w^{i}:=u^{i}+\sqrt{-1}v^{i}$ on $C(V)$. Then, the $(0,2)$-part of $F$ is given by 
\[F^{(0,2)}=\frac{1}{2}\sum_{1\leq i<\ell \leq k}\left(\frac{\partial \langle \tilde{\Theta}(\xi_{\ell}),Y\rangle}{\partial u^{i}}-\frac{\partial \langle \tilde{\Theta}(\xi_{i}),Y\rangle}{\partial u^{\ell}}\right)d\bar{w}^{i}\wedge d\bar{w}^{\ell}. \]
Hence, the condition $F^{(0,2)}=0$ is equivalent to that the 1-form $\eta'$ defined by
\[\eta':=-\sum_{j=1}^{k}\langle \tilde{\Theta}(\xi_{j}),Y\rangle du^{j}\]
is closed. 
Here, note that $\eta'$ is nothing but $\eta(\,\cdot\,)=\omega(Y, \,\cdot\,)$ defined in the proof of Proposition \ref{Lagpt}. 
Actually, from \eqref{etaj} it follows that 
\[
\begin{aligned}
\eta_{j}=&-\sum_{i,\ell=1}^{m}\frac{\partial^2 K_{\lambda}}{\partial x^{i}\partial x^{\ell}}Y^{i}\frac{\partial x^{\ell}}{\partial u^{j}}
=-\sum_{i=1}^{m}\frac{\partial}{\partial u^{j}}\frac{\partial K_{\lambda}}{\partial x^{i}}Y^{i}\\
=&-\sum_{i=1}^{m}\frac{\partial \tilde{x}_{i}}{\partial u^{j}}Y^{i}
=-\sum_{i=1}^{m}\langle \xi_{j},e^{i}\rangle Y^{i}
=-\langle \tilde{\Theta}(\xi_{j}),Y\rangle, 
\end{aligned}
\]
where we used the relation $\tilde{x}_{i}=\partial K_{\lambda}/\partial x^{i}$ at the third equality and used \eqref{CR} at the last equality. 
Thus, the integrability of $D^{Y}$ is equivalent to the closedness of $\omega(Y, \,\cdot\,)$, 
and the latter is equivalent to the existence of a local potential of $Y$. 
\end{proof}

Next, we consider the deformed Hermitian Yang--Mills condition for $D^{Y}$. 
This is a mirror side proposition which corresponds to Proposition \ref{sLagpt}. 

\begin{proposition}\label{dHYMprop}
Assume that the connection $D^{Y}$ of the trivial line bundle $\underline{\mathbb{C}}$ over $C(V)$ is integrable. 
Then, it is a deformed Hermitian Yang--Mills connection with phase $e^{\sqrt{-1}\theta_{0}}$
if and only if $V$ and $Y$ satisfy 
\[\arg\det\left[\mathcal{W}(u)|\mathcal{Z} \right]=\theta_{0}\]
for each natural parametrization $u$, see \eqref{natpara} and \eqref{WZ}. 
\end{proposition}
\begin{proof}
Let $(u,v)\mapsto (\tilde{x}(u),\tilde{y}(v))$ be holomorphic coordinates on $C(V)$ taken in the proof of Lemma \ref{lem3}. 
Note that by \eqref{CR} these coordinates satisfy
\[d\tilde{x}_{i}=\sum_{\ell=1}^{k}\langle \xi_{\ell},e^{i}\rangle du^{\ell}\quad\text{and}\quad d\tilde{y}_{j}=\sum_{\ell=1}^{k}\langle \xi_{\ell},e^{j}\rangle dv^{\ell}. \]
From \eqref{curvature} with the condition $F^{(0,2)}=0$, it follows that
\[F=F^{(1,1)}=-\frac{1}{2}\sum_{i,j=1}^{k}\frac{\partial \langle \tilde{\Theta}(\xi_{j}),Y \rangle}{\partial u^{i}}dw^{i}\wedge d\bar{w}^{j},\]
where $w^{i}=u^{i}+\sqrt{-1}v^{i}$. 
By a straightforward computation, the restriction of $\tilde{\omega}$ to $C(V)$ is given by 
\[
\tilde{\omega}=\sum_{i,j=1}^{m}K_{\lambda}^{ij}d\tilde{x}_{i}\wedge d \tilde{y}_{j}
=\frac{\sqrt{-1}}{2}\sum_{i,j=1}^{m}\sum_{\ell,p=1}^{k}K_{\lambda}^{ij}\langle \xi_{\ell},e^{i}\rangle \langle\xi_{p},e^{j}\rangle dw^{\ell}\wedge d \bar{w}^{p}. 
\]
This yields that 
\[
\tilde{\omega}+F=\frac{\sqrt{-1}}{2}\sum_{i,j=1}^{k}\left(\sum_{\ell,p=1}^{m}K_{\lambda}^{p\ell} \langle \xi_{i},e^{p}\rangle \langle\xi_{j},e^{\ell}\rangle +\sqrt{-1}\frac{\partial \langle \tilde{\Theta}(\xi_{j}),Y \rangle}{\partial u^{i}} \right)dw^{i}\wedge d\bar{w}^{j}
\]
on $C(V)$. 
Furthermore, it is easy to see that
\[\frac{\partial x^{\ell}}{\partial u^{i}}=\sum_{p=1}^{m}\frac{\partial x^{\ell}}{\partial \tilde{x}_{p}}\frac{\partial \tilde{x}_{p}}{\partial u^{i}}
=\sum_{p=1}^{m}K_{\lambda}^{p\ell} \langle \xi_{i},e^{p}\rangle
\quad\text{and}\quad
\frac{\partial \langle \tilde{\Theta}(\xi_{j}),Y \rangle}{\partial u^{i}}=\sum_{\ell=1}^{m}\langle \xi_{j},e^{\ell}\rangle \frac{\partial Y_{\ell}}{\partial u^{i}},
\]
where $Y_{i}$ is the coefficient of $\partial/\partial x^{i}$ of $Y$. 
Combining the above equalities gives 
\[
\begin{aligned}
\tilde{\omega}+F=&
\frac{\sqrt{-1}}{2} \sum_{i,j=1}^{k}\sum_{\ell=1}^{m}\langle\xi_{j},e^{\ell}\rangle
\left(\frac{\partial x^{\ell}}{\partial u^{i}}+\sqrt{-1}\frac{\partial Y_{\ell}}{\partial u^{i}} \right)dw^{i}\wedge d\bar{w}^{j}\\
=&\frac{\sqrt{-1}}{2}\sum_{i,j=1}^{k}\sum_{\ell=1}^{m} \langle\xi_{j},e^{\ell}\rangle c_{i}^{\ell}(u)dw^{i}\wedge d\bar{w}^{j}. 
\end{aligned}
\]
Thus, we have
\[(\tilde{\omega}+F)^{k}=\left(\frac{\sqrt{-1}}{2}\right)^{k}k!(-1)^{k(k-1)/2}\det(b_{ij}(u))dW\wedge d\bar{W}, \]
where $b_{ij}(u):=\sum_{\ell=1}^{m}\langle\xi_{j},e^{\ell}\rangle c_{i}^{\ell}(u)$ and $dW:=dw^{1}\wedge\dots \wedge dw^{k}$. 
Put $B(u):=(b_{ij}(u))$ and 
\[
\mathcal{X}:=
\left[
\begin{array}{ccc}
  \langle \xi_{1},e^{1}\rangle  &  \cdots   &  \langle \xi_{k},e^{1}\rangle\\
  \langle \xi_{1},e^{2}\rangle  &  \cdots   &  \langle \xi_{k},e^{2}\rangle\\
  \vdots         &  \ddots   &  \vdots      \\
 \langle \xi_{1},e^{m}\rangle  &  \cdots   &  \langle \xi_{k},e^{m}\rangle
\end{array}
\right]
.\]
Noting that 
\[{}^{t}[\mathcal{X}|\mathcal{Z}]
[\mathcal{W}(u)|\mathcal{Z} ]
=\left[
\begin{array}{cc}
B(u)  & O\\
{}^{t}\mathcal{Z}\mathcal{W}(u) & {}^{t}\mathcal{Z}\mathcal{Z}
\end{array}
\right], \]
we have
\begin{equation}\label{detbij}
\det B(u)=(\det({}^{t}\mathcal{Z}\mathcal{Z}))^{-1}
\det\left[\mathcal{X}|\mathcal{Z} \right] 
\det\left[\mathcal{W}(u)|\mathcal{Z} \right].
\end{equation}
Since 
\[\frac{(\tilde{\omega}+F)^{k}}{\det B(u)}=\left(\frac{\sqrt{-1}}{2}\right)^{k}k!(-1)^{k(k-1)/2}dW\wedge d\bar{W}\]
and the right hand side is a real form, 
the condition such that $D^{Y}$ is a deformed Hermitian Yang-Mills connection with phase $e^{\sqrt{-1}\theta_{0}}$, that is, 
\[\mathop{\mathrm{Im}}\left(e^{-\sqrt{-1}\theta_{0}}(\tilde{\omega}+F)^{k}\right)=0\]
is equivalent to $\arg\det B(u)=\theta_{0}$. 
Furthermore, by \eqref{detbij}, this is equivalent to $\arg\det\left[\mathcal{W}(u)|\mathcal{Z} \right]=\theta_{0}$. 
Thus, the proof is complete. 
\end{proof}

\begin{proof}[Proof of Theorem \ref{main}]
The first statement immediately follows from Proposition \ref{Lagpt} and \ref{integpt}, 
and the second statement also immediately follows from Proposition \ref{sLagpt} and \ref{dHYMprop}. 
\end{proof}


\begin{thebibliography}{9}

\bibitem{Chan}
K. Chan. 
Homological mirror symmetry for $A_n$-resolutions as a $T$-duality. 
{\em J. Lond. Math. Soc. (2)} {\bf 87} (2013), no. 1, 204--222.

\bibitem{CollinsJacobYau}
T. C. Collins, A. Jacob and S.-T. Yau. 
(1,1) forms with specified Lagrangian phase: a priori estimates and algebraic obstructions. 
{\em \rm{arXive:1508.01934}}, 2015.

\bibitem{CollinsXieYau}
T. C. Collins, D. Xie and S.-T. Yau. 
The deformed Hermitian-Yang-Mills equation in geometry and physics. 
{\em \rm{arXive:1712.00893}}, 2017.

\bibitem{Gross}
M. Gross, 
Mirror symmetry and the Strominger-Yau-Zaslow conjecture, 
{\em Current developments in mathematics 2012, 133--191, Int. Press, Somerville, MA}, 2013.

\bibitem{HarveyLawson}
R. Harvey and H.~B. Lawson, Jr. 
Calibrated geometries. 
{\em Acta Math.}, 148:47--157, 1982.

\bibitem{JacobYau}
A. Jacob and S.-T. Yau. 
A special Lagrangian type equation for holomorphic line bundle. 
{\em Math. Ann.} 369 (2017), no. 1-2, 869--898. 

\bibitem{LeungYauZaslow}
N.-C. Leung, S.-T. Yau and E. Zaslow. 
From special Lagrangian to Hermitian-Yang-Mills via Fourier-Mukai transform. 
{\em Adv. Theor. Math. Phys.} {\bf 4} (2000), no. 6, 1319--1341. 

\bibitem{MarinoMinasianMooreStrominger}
M. Mari\~no, R. Minasian, G. Moore and A. Strominger. 
Nonlinear instantons from supersymmetric $p$-branes. 
{\em J. High Energy Phys.} 2000, no. 1, Paper 5, 32 pp. 

\bibitem{Pingali}
V. P. Pingali. 
A note on the deformed Hermitian Yang--Mills PDE. 
{\em \rm{arXive:1509.00943}}, 2016.

\bibitem{StromingerYauZaslow}
A. Strominger, S.-T. Yau, and E. Zaslow. 
Mirror symmetry is {$T$}-duality. 
{\em Nuclear Phys. B}, 479(1-2):243--259, 1996.

\end{thebibliography}
\end{document}